\documentclass[10pt]{article}

\usepackage[margin=2.5cm]{geometry}
\usepackage{lmodern}
\usepackage[T1]{fontenc}
\usepackage[utf8]{inputenc}
\usepackage{microtype}

\usepackage{titlesec}
\usepackage{indentfirst}

\titleformat{\section}
  {\large\bfseries}
  {\thesection}{0.5em}{}

\titleformat{\subsection}
  {\normalsize\bfseries}
  {\thesubsection}{0.5em}{}

\usepackage{etoolbox}

\renewenvironment{abstract}
{
\par\vspace{0.5em}
\noindent\textbf{Abstract --}
\bfseries
}
{\par\vspace{1em}}

\usepackage{cite}
\usepackage{amsmath,amssymb,amsfonts,amsthm,bm,mathtools}
\usepackage{algorithm}
\usepackage{algpseudocode}
\usepackage{graphicx}
\usepackage{textcomp}
\usepackage{xcolor}
\usepackage{subcaption}
\usepackage{url}
\usepackage{balance}
\usepackage{enumitem}
\usepackage{tikz}
\usetikzlibrary{spy}
\usetikzlibrary{calc}

\newtheorem{theorem}{Theorem}
\newtheorem{lemma}[theorem]{Lemma} 
\newtheorem{proposition}[theorem]{Proposition}
\newtheorem{definition}[theorem]{Definition}
\newtheorem{assumption}[theorem]{Assumption}

\newtheorem{remark}[theorem]{Remark}

\newcommand{\N}{\mathbb{N}}
\newcommand{\R}{\mathbb{R}}

\newcommand{\blue}{\color{black}}
\newcommand{\A}{\mathcal{A}}

\newcommand{\KL}{\textnormal{KL}}

\newcommand*{\inte}{\operatorname{int}}
\newcommand*{\dom}{\operatorname{dom}}

\DeclareMathOperator*{\argmin}{arg\,min}

\mathtoolsset{showonlyrefs}

\newcommand{\syst}[1]{%
\left\{\begin{array}{l}
#1
\end{array}\right.
\kern-\nulldelimiterspace}

\newcommand{\sego}[1]{\textcolor{black}{#1}}

\newcommand{\modif}[1]{{\color{black}#1}}

\makeatletter
\renewcommand{\maketitle}{
\begin{center}
    {\LARGE \bfseries \@title \par}
    \vspace{0.8em}
    {\large \@author \par}
    \vspace{1em}
    \rule{\textwidth}{0.5pt}
\end{center}
\vspace{1.5em}
}
\makeatother

\title{Reinterpreting EMML as Mirror Descent for Constrained Maximum Likelihood Estimation}

\author{
\begin{center}
{\large Antonin Clerc$^{1,2}$, Ségolène Martin$^{1}$, Nicolas Papadakis$^{2}$, Gabriele Steidl$^{1}$} \\[0.5em]
{\small
$^1$ Institute of Mathematics, Technische Universität Berlin, Germany \\ 
$^2$ Univ. Bordeaux, CNRS, Bordeaux INP, IMB, UMR 5251, F-33400 Talence, France \\[0.5em]
\texttt{antonin.clerc@math.u-bordeaux.fr, nicolas.papadakis@math.u-bordeaux.fr} \\[-0.5em]
\texttt{martin@math.tu-berlin.de, steidl@math.tu-berlin.de}
}
\end{center}
}

\begin{document}

\maketitle

\begin{center}
\begin{minipage}{0.80\textwidth}
\begin{abstract}
The Expectation--Maximization Maximum Likelihood (EMML) algorithm belongs to the Expectation--Maximization family and is widely used for image reconstruction problems under Poisson noise.
In this paper, we reinterpret EMML as a mirror descent method applied to a reparametrized objective function. 
This perspective allows us to incorporate convex constraints into the algorithm through appropriately chosen Bregman projections, while preserving the multiplicative structure of the EMML updates to ensure computational efficiency. 
We then establish the convergence of the resulting algorithm toward a solution of the constrained maximum-likelihood problem. 
Numerical experiments on hyperspectral unmixing problems demonstrate that the constrained EMML converges in fewer iterations than the classical EMML.
\end{abstract}
\vspace{2em}
\end{minipage}
\end{center}

\section{Introduction}
Image reconstruction problems under Poisson noise naturally arise in scientific imaging domains, where measurements are photon-limited. Prominent examples include fluorescence microscopy~\cite{sarder_deconvolution_2006, bertero_image_2009}, astronomical imaging~\cite{murtagh2006astronomical, dupe_proximal_2009}, hyperspectral image restoration~\cite{zou_restoration_2018}, and Positron Emission Tomography~\cite{shepp1982maximum}. In these applications, photon counts are inherently low and follow discrete statistics.

Mathematically, the goal is to recover an unknown non-negative image $x \in \R^n_{> 0}$ from Poisson-degraded photon count measurements $y \in \R^m_{\geq0}$, modeled as
\begin{equation}\label{eq:poisson_model}
y \sim \mathrm{Poisson}(A x),
\end{equation}
$A \in \R^{m \times n}_{\ge 0}$ is a known degradation operator, typically ill-conditioned. 
Under Poisson noise, the negative log-likelihood of the observation $y$ given an image $x$ is $F(x)=\mathrm{KL}(y, Ax)$ where $\mathrm{KL}$ is the generalized Kullback–Leibler divergence defined for all $u \in \mathbb{R}^m_{\ge 0}$ and $v \in \mathbb{R}^m_{> 0}$ 
\begin{equation}
\mathrm{KL}(u, v) = \sum_{i=1}^m \left[ u_i \log \frac{u_i}{v_i} + v_i - u_i \right],
\end{equation}
with the convention $0 \log 0 = 0$. To ensure that forall $i\in \{1, \dots, m\}$, $(A x)_i > 0$, we assume that each row $a_i$ of $A$ satisfies $a_i \in \R^n_{\ge 0} \setminus \{\mathbf{0}\}$, thereby ruling out the degenerate null Poisson distribution.
This condition is generally satisfied by standard operators such as convolutions with smoothing filters.

Minimizing the Poisson log-likelihood is numerically challenging due to the lack of Lipschitz continuity of the gradient and the singular behavior near the boundary of $\R^n_{> 0}$, which prevents the direct use of standard gradient descent methods. Early iterative reconstruction methods minimizing the KL divergence were introduced in several works~\cite{lucy_iterative_1974, richardson_bayesian-based_1972, shepp1982maximum, byrne1993iterative} based on the Expectation--Maximization (EM) framework. 
In particular, the Richardson--Lucy algorithm~\cite{lucy_iterative_1974, richardson_bayesian-based_1972} was originally proposed for linear image deconvolution under Poisson noise. Later, Byrne~\cite{byrne1993iterative} and Shepp and Vardi~\cite{shepp1982maximum} independently derived the same multiplicative updates within the EM formalism, introducing the Expectation–Maximization Maximum Likelihood (EMML) and Maximum Likelihood Expectation–Maximization (MLEM) algorithms, respectively. Although introduced in different communities under different names, these methods correspond to the same EM iterations for Poisson maximum-likelihood estimation. In the remainder of this work, we  refer to all these algorithms as EMML.

The EMML iterations to minimize $F$ are given by
\begin{equation}\label{eq:EMML_update}
x^{(k+1)} = x^{(k)} \odot \frac{A^\top \frac{y}{A x^{(k)}}}{A^\top \mathbf{1}},
\end{equation}
where multiplications and divisions are component-wise, and $\odot$ denotes the Hadamard product. These iterations are simple, preserve non-negativity and monotonically increase the likelihood \cite{byrne1993iterative}. EMML typically provides high-quality reconstruction after only a few iterations, but the quality may degrade afterward. 
Nevertheless, many works continue to rely on EMML ~\cite{zunino_reconstructing_2023,zhang2025probing}.
To improve reconstruction quality and mitigate the need for early stopping, recent approaches extend the EMML iterations by incorporating regularization, see, e.g.~\cite{SST2010,Kasai2025, Deidda2019HKEM, marquis2023pet,modrzyk2024fast, modrzyk2026convergentplugandplaymajorizationminimizationalgorithm}, or use them to initialize a neural network~\cite{daniele2025deepequilibriummodelspoisson}. 

The focus of this paper is on incorporating constraints into the EMML algorithm while preserving its fast multiplicative updates. Many imaging problems naturally involve the minimization of a convex objective under convex constraints \cite{afonso2010augmented, ha_estimating_2019}. Yet, standard EMML iterations do not directly accommodate these constraints, motivating the developments proposed in this work.
Existing constrained variants of EMML, such as \cite{constrainedEMML}, only account for averages over selected pixels, corresponding to constraints of the form $\sum_{i \in \mathcal{I}} x_i = y$, and do not address general convex constraints. 
More broadly, the EM framework is not naturally tailored to constrained convex optimization problems. By contrast, first-order optimization methods provide a systematic framework for handling general convex constraints. This limitation motivates the search for a more flexible optimization framework, which we obtain by revisiting EMML from a first-order optimization perspective.

\paragraph*{Contributions} 
Building on the connection between EM algorithms for the exponential family and mirror descent (MD), we reinterpret EMML as a particular instance of MD. This perspective naturally accommodates general convex constraints while preserving both convergence guarantees and the characteristic multiplicative structure of the updates.

Section~\ref{sec:MD} recalls MD results. Section~\ref{sec:EMML_as_a_MD} shows that EMML can be interpreted as MD on a reparametrized objective. Section~\ref{sec:Constrained_EMML} extends EMML to constrained problems using Bregman projections, preserving convergence guarantees. Section~\ref{sec:numerical_results} presents numerical experiments on hyperspectral image restoration under Poisson noise, demonstrating improved PSNR and faster convergence with the constrained EMML. Conclusions are drawn in Section~\ref{sec:Conclusion}.

\section{Background on Mirror Descent} \label{sec:MD}

In this section, we revisit several convergence results for Mirror Descent (MD). Our analysis is closely related to that of \cite{bolte2018first}, but requires a slight extension of their framework to accommodate a limiting stepsize regime. More precisely, we relax a strict stepsize condition appearing in the standard theory to its boundary case. For completeness, we therefore restate the main assumptions and derive the convergence results that will be needed in the remainder of the paper.

Let $F: \R^n \to \R \cup \{+\infty\}$ \modif{such that $F(x)=f(x)+g(x)$}. Let $u: \R^n \to \R \cup \{+\infty\}$ be a strictly convex and differentiable function on $\inte(\dom u)$ where $\dom u = \{x \in \R^n : u(x) < +\infty\}$. The mapping $\nabla u$ is referred to as a \textit{mirror map}. 
For all $(x,y) \in \dom u \times \inte(\dom u)$, the Bregman divergence associated with the potential $u$ is defined by
\begin{equation*}
D_u(x,y) = u(x) - u(y) - \langle \nabla u(y), x-y \rangle.
\end{equation*}
Given a step size $\tau>0$, and $g$ differentiable on $\inte \dom u$, we define the operator
\begin{equation} \label{eq:T_tau_x}
T_\tau(x)
=
\argmin_{z \in \R^n}
\left\{
\modif{f(z)}
+
\langle z-x,\nabla g(x)\rangle
+
\frac{1}{\tau}D_u(z,x)
\right\}.
\end{equation}
The Bregman proximal gradient (BPG) iterates $\{x^k\}_{k\in\N}$ are then defined by 
\begin{equation}\label{eq:MD_Bregman_prox}
x^{(k+1)} \in T_\tau(x^{(k)}),
\end{equation}
see \cite{Teboulle2018}.
MD applied to the minimization of $g$ is recovered as a special case of the BPG when $f\equiv0$. The first-order optimality conditions for \eqref{eq:MD_Bregman_prox} then yield the classical MD update
\begin{equation}\label{eq:MD_update}
x^{(k+1)}
=
\nabla u^*
\left(
\nabla u(x^{(k)})
-
\tau \nabla g(x^{(k)})
\right),
 \end{equation}
where $u^*$ denotes the convex conjugate of $u$ \cite{rockafellar_convex_1997}. To establish convergence of the iterates when $F$ is neither convex nor Lipschitz continuous, we impose the following assumptions:
\begin{assumption}\label{assumption:MD_nonLipschitz_1}~
\begin{enumerate}[leftmargin=*, itemsep=0.5em]

\item $u$ is of Legendre type. 

\modif{
\item $g$ is proper, lower semicontinuous, with $\dom u \subseteq \dom g$, and continuously differentiable on $\inte(\dom u)$. Moreover, $g$ is relatively smooth with respect to $u$, i.e., there exists $L>0$ such that $Lu - g$ is convex on $\inte(\dom u)$.

\item $f$ is proper, lower semicontinuous, and convex, with $\dom u \cap \dom f \neq \emptyset$.
}

\item
$
v(\mathcal P) := \inf\{F(x): x \in \overline{\dom}\, u\}
$
satisfies
$
v(\mathcal P) > -\infty.
$ 
\item The operator \eqref{eq:T_tau_x} is nonempty, single-valued, and maps $\inte(\dom u)$ into $\inte(\dom u)$.
\end{enumerate}
\end{assumption}

Assumption~\ref{assumption:MD_nonLipschitz_1}
 is satisfied for 
$g(x) = \mathrm{KL}(y, Ax),~f(x) = 0$ with Burg's entropy~\cite{daniele2025deepequilibriummodelspoisson, bauschke2017descent} or Shannon entropy~\cite{kunstner2021homeomorphic} as the mirror map, provided that $A^\top y > 0$ to verify point~\textit{5)}. We next recall convergence results for MD following \cite[Remark 4.1, Proposition~4.1]{bolte2018first}.

\modif{
\begin{proposition}
\label{prop:cvMD_iter}
Suppose that Assumption \ref{assumption:MD_nonLipschitz_1} is fulfilled.
Let $(x^{(k)})_{k\in\mathbb{N}}$ be the sequence generated by \eqref{eq:MD_Bregman_prox} with stepsize  $\tau = \frac{1}{L}$. 
Then the sequence $\bigl(F(x^{(k)})\bigr)_{k\in\mathbb{N}}$ is non-increasing. 
and we have for all $N \in \N$ that
\begin{equation}
\min_{1\le k\le N} D_u(x^{(k)}, x^{(k+1)})
\le
\frac{\tau}{N}\bigl(F(x^{(0)}) - \inf F\bigr).
\end{equation}
\end{proposition}

\begin{proof} ~

We first recall the following Lemma adapted from \cite[Remark~4.1]{bolte2018first}:
\begin{lemma}
\label{lemma:inequality_MD}
Suppose that Assumption~\ref{assumption:MD_nonLipschitz_1} holds. Let $\tau\in(0,1/L]$.
For $x\in \inte(\dom u)\cap\dom f$, set $x^+\coloneqq T_\tau(x)$. Then, for all 
$s\in \dom f\cap\dom u$,
\begin{equation}
\label{eq:ineq_nonconvex_lemma}
\tau\bigl(F(x^+)-F(s)\bigr)
\le
D_u(s,x)-D_u(s,x^+)-\tau D_g(s,x).
\end{equation}
In particular, taking $s=x$ gives
\begin{equation}
\label{eq:ineq_nonconvex}
\tau\bigl(F(x^+)-F(x)\bigr)
\le
-D_u(x,x^+).
\end{equation}
\end{lemma}

We can now prove the main claims of the proposition.
Let $k \ge 0$. Applying Lemma~\ref{lemma:inequality_MD} with $x=x^k,~ x^+=T_\tau(x^k)=x^{k+1}$, and using the non-negativity of the Bregman divergence, we obtain
\begin{equation}
\label{eq:ineq_nonconvex_proof}
\tau\bigl(F(x^{k+1}) - F(x^k)\bigr)
\le
- D_u(x^k,x^{k+1})
\le 0 .
\end{equation}
Hence the sequence $\bigl(F(x^k)\bigr)_{k\in\N}$ is non-increasing. Summing \eqref{eq:ineq_nonconvex_proof} from $k=0$ to $n-1$ gives
\begin{align}
\sum_{k=0}^{n-1} D_u(x^k,x^{k+1})
&\le
\tau \sum_{k=0}^{n-1} \bigl(F(x^k)-F(x^{k+1})\bigr) \\
&=
\tau \bigl(F(x^0)-F(x^n)\bigr) \\
&\le
\tau \bigl(F(x^0)-\inf F\bigr).
\label{eq:ineq_nonconvex_align}
\end{align}
Here we used that $F(x^n)\geq \inf F$ and that $\inf F>-\infty$ by Assumption~\ref{assumption:MD_nonLipschitz_1}.
From \eqref{eq:ineq_nonconvex_align}, we have
\begin{equation}
n \min_{0\le k\le n-1} D_u(x^k,x^{k+1})
\le
\tau \bigl(F(x^0)-\inf F\bigr).
\end{equation}
Dividing by $n$ yields the desired result.
\end{proof}

To establish the global convergence of the iterates, we largely follows that of \cite{bolte2018first}. However, we extend \cite[Theorem~4.1]{bolte2018first} to the limiting stepsize regime $\tau=\frac{1}{L}$ by relaxing a strict stepsize condition to a non-strict one.
We further introduce the following assumptions.
\begin{assumption}\label{assumption:MD_nonLipschitz_2}~
\begin{enumerate}[leftmargin=*, itemsep=0.5em]
\item$\dom u = \mathbb{R}^n$ and, on every bounded subset of $\mathbb{R}^n$, $u$ is $\sigma$-strongly convex.

\item $\nabla u$ and $\nabla g$ are Lipschitz continuous on every bounded subset of $\mathbb{R}^n$.
\end{enumerate}
\end{assumption}

These two assumptions differ from those in \cite[Assumption D]{bolte2018first}, but they are satisfied by the specific choice of function we adopt for EMML, which does not satisfy the $\sigma$-strong convexity of $u$ on all $\R^n$ \cite[Assumption~D]{bolte2018first}. However, we can recover similar convergence results as in \cite{bolte2018first} assuming that iterates stay bounded by example. 
We also recall the following definition from \cite[Definition~4.1]{bolte2018first} we will use to prove the convergence of the iterates.
\begin{definition}\label{def:gradient_like_seq}
\blue
Let $ F : \mathbb{R}^d \to (-\infty, \infty] $ be a proper and lower semicontinuous function. 
A sequence $ (x_k)_{k \in \mathbb{N}} $ is called a gradient-like descent sequence for $ F $ if the following three conditions hold:

\textbf{(C1) Sufficient decrease property.} There exists a constant $ \rho_1 > 0 $ such that
\begin{equation}
\rho_1 \, \|x_{k+1} - x_k\|^2 \leq F(x_k) - F(x_{k+1}), \quad \forall k \in \mathbb{N}.
\end{equation}

\textbf{(C2) Subgradient lower bound for the iterates gap.} There exists a constant $ \rho_2 > 0 $ such that
\begin{equation}
\|w_{k+1}\| \leq \rho_2 \, \|x_{k+1} - x_k\|, \quad \text{for some } w_{k+1} \in \partial F(x_{k+1}), \ \forall k \in \mathbb{N}.
\end{equation}

\textbf{(C3) Lower semicontinuity along limit points.} If $ x $ is a limit point of a subsequence $ (x_k)_{k \in K} $, with $ K \subset \mathbb{N} $, then
\begin{equation}
\limsup_{k \in K} F(x_k) \leq F(x).
\end{equation}
\end{definition}

A straightforward adaptation of \cite[Theorem~4.1]{bolte2018first} for the limiting case $\tau=\frac{1}{L}$ and the relaxed Assumption \ref{assumption:MD_nonLipschitz_2} gives the following proposition.
\begin{proposition} \label{prop:cvMD_sequence}
Suppose that Assumptions \ref{assumption:MD_nonLipschitz_1} and \ref{assumption:MD_nonLipschitz_2} are fulfilled.
Assume that the sequence $(x^{(k)})_{k\in\mathbb{N}}$ generated by \eqref{eq:MD_Bregman_prox} with stepsize $\tau = \frac{1}{L}$ is bounded.
Then the following holds true:

    i) \textbf{Subsequential convergence.}
    Any limit point of the sequence $( x^{(k)})_{k\in\mathbb{N}}$ is a critical point of $F$.
    
    ii) \textbf{Global convergence.}
    Suppose that $F$ satisfies the Kurdyka--Łojasiewicz property on $\operatorname{dom}F$. Then $( x^{(k)})_{k\in\mathbb{N}}$ has finite length and converges to a critical point $x^\ast$ of $F$.
\end{proposition}

\begin{proof}~
\blue
\begin{enumerate}[leftmargin=*, itemsep=0.5em]
    \item We first show that the sequence $(x_k)_{k \in \mathbb{N}}$ is a gradient-like descent sequence by verifying that conditions C1, C2, and C3 of Definition~\ref{def:gradient_like_seq} are satisfied.
    \paragraph{(C1) Sufficient decrease property}
    From Proposition \ref{prop:cvMD_iter} and Assumption \ref{assumption:MD_nonLipschitz_2}, there exists $\sigma >0$ such that
    \begin{equation} \label{eq:strongconv_ineq}
    \tau\bigl(F(x^{k}) - F(x^{k+1})\bigr)
    \ge
    D_u(x^k,x^{k+1})
    \ge
    \frac{\sigma}{2}\|x^{k+1}-x^{k}\|^2,
    \end{equation}
    where the last inequality follows from the $\sigma$-strong convexity of $u$ on compact sets. Note that \eqref{eq:strongconv_ineq} also implies that $\|x^{k+1}-x^{k}\| \underset{k \rightarrow +\infty}{\to} 0$ because $D_u(x^k,x^{k+1})\underset{k \rightarrow +\infty}{\to} 0$ from Proposition \ref{prop:cvMD_iter}.
    This proves that condition C1 holds true.

    \paragraph{(C2) Subgradient lower bound for the iterates gap}
    Writing the optimality condition of the optimization problem \eqref{eq:MD_Bregman_prox} which defines $x^{k+1}$ yields:
    \begin{equation}
        0 \in \partial f(x^{k+1}) + \nabla g(x^k) + \frac{1}{\tau} \bigl(\nabla u(x^{k+1}) - \nabla u(x^k)\bigr).
    \end{equation}
    Moreover we have that $\partial F(x) = \partial f(x) + \nabla g(x)$ for all $x \in \dom u$. Let for all $k \in \mathbb{N}$
    \begin{equation}
        w^{k+1} := \nabla g(x^{k+1}) - \nabla g(x^{k}) + \frac{1}{\tau} \bigl(\nabla u(x^{k}) - \nabla u(x^{k+1})\bigr).
    \end{equation}
    Then $w^{k+1} \in  \nabla g(x^{k+1}) + \partial f(x^{k+1}) \subset \partial F(x^{k+1})$.
    Since $(x_k)_{k \in \mathbb{N}}$ is a bounded sequence and both $\nabla u$ and $\nabla g$ are Lipschitz continuous on bounded subsets of $\mathbb{R}^d$ (see Assumption \ref{assumption:MD_nonLipschitz_2}), there exists $M > 0$ such that
    \begin{align}
        \|w^{k+1}\|
        &\le
        \|\nabla g(x^{k+1}) - \nabla g(x^{k})\|
        +
        \frac{1}{\tau}\|\nabla u(x^{k}) - \nabla u(x^{k+1})\|\\
        &\le
        M\left(1+\frac{1}{\tau}\right)\|x^{k+1}-x^{k}\|.
    \end{align}
    This proves that condition C2 also holds true.
    
    \paragraph{(C3) Lower semicontinuity along limit points}
    Consider a subsequence $(x^{n_k})_{k \in \mathbb{N}}$ which converges to some point $x^\ast$. Note that there exists such a subsequence since the sequence $(x^k)_{k \in \mathbb{N}}$ is assumed to be bounded. Since $\Vert x^{k} - x^{k-1} \Vert \to 0$, the sequence $(x^{n_k-1})_{k \in \mathbb{N}}$ also converges to $x^\ast$.
    In addition, since $u$ is continuously differentiable on $\mathbb{R}^d$, we have
    \begin{equation} \label{eq:D_x_xstar}
    \lim_{k \to \infty} D_u(x^\ast, x^{n_k-1}) = 0.
    \end{equation}
    Now from $x^{(k+1)} = T_\tau(x^k)$ given by \eqref{eq:T_tau_x} it follows that:
    \begin{align}
    f(x^k) + \langle x^k-x^{k-1}, \nabla g(x^{k-1}) \rangle + \frac{1}{\tau} D_u(x^k, x^{k-1})
    &\leq
    f(x^\ast) + \langle x^\ast-x^{k-1}, \nabla g(x^{k-1}) \rangle + \frac{1}{\tau} D_u(x^\ast, x^{k-1}), \end{align}
    thus
    \begin{align}
    f(x^k) &\leq f(x^\ast)
    + \langle x^\ast - x^k, \nabla g(x^{k-1}) \rangle
    + \frac{1}{\lambda} D_u(x^\ast, x^{k-1})
    - \frac{1}{\lambda} D_u(x^k, x^{k-1})\\
    & \leq f(x^\ast)
    + \langle x^\ast - x^k, \nabla g(x^{k-1}) \rangle
    + \frac{1}{\lambda} D_u(x^\ast, x^{k-1}).
    \end{align}
    
    Substituting $k$ by $n_k$ and letting $k \to \infty$, we obtain from \eqref{eq:D_x_xstar} and lower semi-continuity of $f$
    \begin{equation}
    \limsup_{k \to \infty} f(x^{n_k}) \leq f(x^\ast).
    \end{equation}
    From the continuity of $g$, it follows that   \begin{equation}
    \limsup_{k \to \infty} F(x^{n_k}) \leq F(x^\ast).
    \end{equation}
    This proves condition C3, and thanks to Lemma~\cite[Lemma~4.2]{bolte2018first}, the first item of the proposition follows. 
    We thus obtain that $F$ is constant on the set of all cluster points of the sequence $(x^k)_{k \in \mathbb{N}}$ by \cite[Lemma~4.2]{bolte2018first}.
    \item The second part of the proof stay unchanged w.r.t. the proof of~\cite[Theorem~4.1]{bolte2018first}
\end{enumerate}
\end{proof}
}

\section{The EMML Algorithm} \label{sec:EMML_as_a_MD}

Following the Poisson model~\eqref{eq:poisson_model},  
the EMML algorithm \eqref{eq:EMML_update} computes the maximum-likelihood estimate of $x$ by minimizing the negative log-likelihood
$
x \mapsto \mathrm{KL}(y, Ax).
$

\subsection{EMML as EM algorithm}

To derive the EM algorithm, we introduce latent variables $ z_{ij} \in \mathbb{N} $ representing the number of photons emitted from pixel $ j $ and detected by detector $ i $ as
\begin{equation}
    z_{ij} \sim \mathrm{Poisson}(A_{ij} x_j), 
    \qquad
    y_i = \sum_{j=1}^{n} z_{ij},
\end{equation}
where the observed counts are obtained by marginalization. The joint complete-data likelihood can then be written as
\begin{equation} \label{eq:complete_loglike}
    p(z,y \mid x)
    = \frac{1}{\prod_{i,j} z_{ij}!}
       \exp\Big(
           \sum_{i,j \sego{:\, A_{ij}>0}} z_{ij} \log(A_{ij} x_j)
           - A_{ij} x_j
       \Big).
\end{equation}

\noindent
We now apply the EM algorithm.

\textbf{E-step:} We compute the expected complete-data negative log-likelihood:
\begin{align*}
    Q_{x^{(r)}}(x)
    &= -\mathbb{E}_{p(z \mid y, x^{(r)})} \big[ \log p(z, y \mid x) \big] \\
    &\propto -\sum_{i, j \sego{:\, A_{ij}>0}}\left(
        -A_{ij}x_j
        + y_i \frac{A_{ij} x_j^{(r)}}{(Ax^{(r)})_i}
        \log(A_{ij}x_j)
    \right).
\end{align*}
This function is the EM surrogate associated with the observed negative log-likelihood, which coincides with $\KL(y,Ax)$ up to a constant independent of $x$.

\textbf{M-step:} We minimize $Q_{x^{(r)}}$ with respect to $x$. Setting the gradient with respect to $x_j$ to zero yields
\begin{align*}
0 = \sum_{i: \sego{A_{ij}>0}} \left(
A_{ij}
- \frac{1}{x_j}
  \frac{y_i A_{ij} x_j^{(r)}}{(Ax^{(r)})_i}
\right).
\end{align*}
Solving for $x_j$ leads to the classical EMML iteration \eqref{eq:EMML_update}.

\subsection{EMML as MD}

EM applied to exponential families is known to admit a MD interpretation \cite{kunstner2021homeomorphic, aubin-frankowski_mirror_2022}.  
Since the Poisson distribution belongs to the exponential family \cite{rossignol_bregman_2022}, this framework applies to our setting.  
In particular, the joint complete-data log-likelihood \eqref{eq:complete_loglike} can be rewritten in the canonical exponential-family form as
\begin{equation}
    p(z,y \mid x)
    = \frac{1}{h(z)}
       \exp\left(
           \langle S(z,y), \theta \rangle
           - \A(\theta)
       \right).
\end{equation}
The corresponding components are:

\begin{itemize}
    \item natural parameter  
    $\theta_j = \log x_j$, $j=1,\dots,n$,
    
    \item sufficient statistic  
    $S(z,y)_j = \sum_{i}  z_{ij}$,
    
    \item log-partition function  
    $\A(\theta)
    = \sum_{i,j} A_{ij} e^{\theta_j}
    = \mathbf{1}^\top A e^{\theta}$,
    
    \item base measure  
        $h(z) = \exp\,(
            -\sum_{i,j \sego{:\, A_{ij}>0}} z_{ij} \log A_{ij}
            + \log(z_{ij}!)
        )$.
\end{itemize}

This formulation is key for interpreting EMML as a MD algorithm. Exploiting the mirror map and the reparametrization of \cite{kunstner2021homeomorphic}, we obtain the following proposition for EMML.

\begin{proposition}\label{prop:EMML_as_MD}
Define $g:\mathbb{R}^n \to \mathbb{R}$ by $g(\theta)=\KL(y, A e^{\theta})$ 
and assume that every row and every column of $A$ contains at least one strictly positive entry, and that $A^\top y >0$. Consider MD applied to $g$ with mirror map
\begin{equation}\label{eq:mirror_map_theta}
\A(\theta)=\sum_{i,j} A_{ij} e^{\theta_j}.
\end{equation}
Then the corresponding MD update (see relation ~\eqref{eq:MD_update})
 \begin{equation}\label{eq:MD_EMML}
\theta^{(r+1)} = (\nabla \A)^{-1}\bigl( \nabla \A(\theta^{(r)}) - \nabla g(\theta^{(r)}) \bigr)
\end{equation}
is equivalent to the classical EMML iteration, up to the reparametrization $x^{(r)} = \exp(\theta^{(r)})$.
\end{proposition}

The convergence of the EMML algorithm was established earlier in \cite{byrne1993iterative, shepp1982maximum}, where the decrease property of the iterates was analyzed using arguments based on the \KL function.
Convergence rates for EM algorithms associated with exponential family models were subsequently analyzed in \cite[Proposition~2]{kunstner2021homeomorphic} and \cite[Proposition~11]{aubin-frankowski_mirror_2022}.

\modif{
\begin{proposition} \label{prop:cv_speed_EMML}
Suppose that iterates $(\theta^{(k)})_k$ in the MD~\eqref{eq:MD_EMML} stay bounded. Then the sequence converges in the sense of Proposition \ref{prop:cvMD_sequence}. Moreover,
for any $N \in \mathbb{N}^*$ and any initial point $x^{(0)} \in \R^n_{>0}$,
there exists a constant $K$ such that the EMML iterates satisfy
\begin{equation}
\min_{1\le k\le N} \|x^{(k)} - x^{(k+1)}\|^2 \leq \frac{K}{N}\bigl(g(\theta^{(0)}) - \inf g\bigr).
\end{equation}
\end{proposition}
}

\begin{proof}~
\blue
Define $f\equiv 0$, $g(\theta) = \KL(y, A e^{\theta})$. We first prove that Assumptions \ref{assumption:MD_nonLipschitz_1} and \ref{assumption:MD_nonLipschitz_2} are verified, and then apply Propositions \ref{prop:cvMD_iter} and \ref{prop:cvMD_sequence}.
\paragraph{Verification of Assumption~\ref{assumption:MD_nonLipschitz_1}}.

\begin{enumerate}
    \item Since every column of $A$ contains at least one strictly positive entry, if we set $a_j\coloneqq \sum_{i=1}^m A_{ij}$, then $a_j>0$ for all $j=1,\dots,n$, and
    \begin{equation}
        \A(\theta)=\sum_{j=1}^n a_j e^{\theta_j}.
    \end{equation}
    Thus $\A$ is proper, lower semicontinuous, and differentiable on $\inte(\dom  \A)=\dom  \A=\R^n$. Moreover, $\nabla^2  \A(\theta)=\operatorname{Diag}\{(a_1e^{\theta_1},\dots,a_ne^{\theta_n})\}\succ 0$ on any bounded subset of $\R^n$, so $ \A$ is strictly convex on $\theta\in\R^n_{>-R}$. Since $\dom  \A=\R^n$ has empty boundary, the boundary condition in the definition of essential smoothness is vacuous. Hence $ \A$ is essentially smooth and strictly convex, and therefore $ \A$ is of Legendre type.

    \item By the row assumption on $A$, $Ae^\theta\in\R^m_{++}$ for every $\theta\in\R^n$. Hence $g(\theta)=\KL(y,Ae^\theta)$ is finite for every $\theta\in\R^n$, so $\dom g=\R^n$. Since $\theta\mapsto Ae^\theta$ is smooth and the KL divergence is smooth in its second argument on $\R^m_{++}$, $g$ is differentiable on $\R^n=\inte(\dom  \A)$. In particular, $g$ is proper and lower semicontinuous, and $\dom  \A\subseteq \dom g$.
    
    We now show that $g$ is $1$-smooth relative to $ \A$, that is, that $ \A-g$ is convex on $\R^n$. Since
    \begin{equation}
\KL(y,z)=\sum_{i=1}^m\left(y_i\log\frac{y_i}{z_i}-y_i+z_i\right).
    \end{equation}
    and $ \A(\theta)=\sum_{i=1}^m(Ae^\theta)_i$, we get
    \begin{equation}
         \A(\theta)-g(\theta)
        =
        \sum_{i=1}^m y_i\log(Ae^\theta)_i
        -
        \sum_{i=1}^m(y_i\log y_i-y_i).
    \end{equation}
    The second term is constant. For each $i$, the function $\theta\mapsto \log(Ae^\theta)_i=\log(\sum_{j=1}^n A_{ij}e^{\theta_j})$ is a log-sum-exp function, hence is convex. Since $y_i\geq 0$ for every $i$, $ \A-g$ is a nonnegative weighted sum of convex functions, up to an additive constant. Therefore $ \A-g$ is convex, so $g$ is $1$-smooth relative to $ \A$.

    \item Since $f\equiv 0$, $f$ is proper, lower semicontinuous, and convex, with $\dom  \A \cap \dom f = \R^n \neq \emptyset$.

    \item As the KL divergence is nonnegative, $g(\theta)\geq 0$ for every $\theta\in\R^n$. Since $f\equiv 0$, we have $F(\theta)\geq 0$ on $\dom F$ and $F(\theta)=+\infty$ outside $\dom F$. Therefore
    \begin{equation}
        v(\mathcal P)=\inf\{F(\theta):\theta\in\overline{\dom}\, \A\}\geq 0>-\infty.
    \end{equation}

    \item In our setting, $\inte(\dom  \A)=\R^n$. We show that, for every 
$\theta\in\R^n$ and every $\tau\in(0,1)$ or ($\tau=1$ and $A^\top y \in \R_{++}^n$), the operator $T_\tau(\theta)$ is nonempty, single-valued, and belongs to $\R^n$.

Let
$
    a \coloneqq A^\top \mathbf 1_m.
$
Since every column of $A$ contains at least one strictly positive entry, we have $a\in\R^n_{++}$ and
\begin{equation}
     \A(z)=\langle a,e^z\rangle,
    \qquad 
    \nabla  \A(\theta)=e^\theta\odot a.
\end{equation}
where $ \A$ is strictly convex on every bounded subset of $\R^n$.
Moreover, since
$
    g(\theta)=\KL(y,Ae^\theta),
$
we have
\begin{equation}
    \nabla g(\theta)
    =
    e^\theta\odot A^\top\left(\mathbf 1_m-\frac{y}{Ae^\theta}\right),
\end{equation}
where the division is componentwise. Hence
\begin{equation}
    \nabla g(\theta)-\frac1\tau \nabla  \A(\theta)
    =
    -e^\theta\odot
    \left[
        \left(\frac1\tau-1\right)A^\top\mathbf 1_m
        +
        A^\top\left(\frac{y}{Ae^\theta}\right)
    \right].
\end{equation}
Set
\begin{equation}
    q_{\tau,\theta}
    \coloneqq
    e^\theta\odot
    \left[
        \left(\frac1\tau-1\right)A^\top\mathbf 1_m
        +
        A^\top\left(\frac{y}{Ae^\theta}\right)
    \right],
\end{equation}
then, for $\tau\in(0,1)$ or ($\tau=1$ and $A^\top y \in \R_{++}^n$), we have $q_{\tau,\theta}\in\R^n_{++}$. Now, up to an additive constant independent of $z$,
\begin{equation} \label{eq:proof_Ttau}
    \langle z,\nabla g(\theta)\rangle
    +
    \frac1\tau D_u(z,\theta)
    =
    \frac1\tau \langle a,e^z\rangle
    -
    \langle q_{\tau,\theta},z\rangle .
\end{equation}
The right-hand side is coercive on $\R^n$. Indeed, the exponential term 
$\langle a,e^z\rangle$ controls the directions where some components of $z$ go to $+\infty$, while the term 
$-\langle q_{\tau,\theta},z\rangle$ controls the directions where some components of $z$ go to $-\infty$, because $q_{\tau,\theta}\in\R^n_{++}$.

Therefore the function
\begin{equation}
    z\mapsto 
    \langle z-\theta,\nabla g(\theta)\rangle
    +
    \frac1\tau D_u(z,\theta)
\end{equation}
is proper, lower semicontinuous and coercive. The minimizer is unique because $z\mapsto D_u(z,\theta)$ is strictly convex by strong convexity of $ \A$ on every bounded subset of $\R^n$, and adding a linear term preserves strict convexity on $C_\theta$. Thus $T_\tau(\theta)$ is nonempty and single-valued.
Finally, since $\dom  \A=\R^n$, the minimizer automatically belongs to $\inte(\dom  \A)=\R^n$. Hence $T_\tau$ maps $\inte(\dom  \A)$ into $\inte(\dom  \A)$.
\end{enumerate}
Thus all points of Assumption~\ref{assumption:MD_nonLipschitz_1} are satisfied.

\paragraph{Verification of Assumption~\ref{assumption:MD_nonLipschitz_2}}.
\begin{enumerate}

\item Let $a\coloneqq A^\top\mathbf 1_m$. Since every column of $A$ contains at least one strictly positive entry, $a\in\R^n_{++}$ and
$
     \A(\theta)=\langle a,e^\theta\rangle.
$
Thus $\dom  \A=\R^n$ and
\begin{equation}
    \nabla^2  \A(\theta)
    =
    \operatorname{diag}(a\odot e^\theta).
\end{equation}
Let $B\subset\R^n$ be bounded. Then there exists $M>0$ such that $|\theta_j|\le M$ for all $\theta\in B$ and all $j$. Hence
\begin{equation}
    \forall \theta \in B, \quad
    \nabla^2  \A(\theta)
    \succeq
    \left(\min_{1\le j\le n} a_j e^{-M}\right) I
\end{equation}
Taking $\sigma=\min_{1\le j\le n} a_j e^{-M}>0$, we deduce that $ \A$ is strongly convex on every bounded subset of $\R^n$.

\item $\nabla  \A$ is Lipschitz continuous on every bounded subset of $\R^n$ because $\nabla^2  \A(\theta)=\operatorname{diag}(a\odot e^\theta)$ is bounded on bounded sets. Moreover, since every row of $A$ contains at least one strictly positive entry, $Ae^\theta\in\R^m_{++}$ for every $\theta\in\R^n$, 
and thus
\begin{equation}
    \nabla g(\theta)
    =
    e^\theta\odot A^\top\left(\mathbf 1_m-\frac{y}{Ae^\theta}\right).
\end{equation}
is well defined. 
Moreover, on every bounded subset of $\R^n$, the vector $e^\theta$ is bounded above and below away from zero, and $Ae^\theta$ is also bounded above and below away from zero componentwise. Hence the Hessian of $g$ is bounded on bounded subsets, which implies that $\nabla g$ is Lipschitz continuous on every bounded subset of $\R^n$.
\end{enumerate}
Thus all points of Assumption~\ref{assumption:MD_nonLipschitz_2} are satisfied.

\paragraph{Convergence of EMML from a MD point of view.}
We now state the convergence result for EMML. First, under the assumption that the iterates are bounded, they converge in the sense of Proposition~\ref{prop:cvMD_sequence}.
Moreover, Proposition \ref{prop:cvMD_iter} yields
\begin{equation}
\min_{1\le k\le n} D_\A(\log(x^{(k)}), \log(x^{(k+1)}))
\le
\frac{\tau}{n}\bigl(g(\theta^{(0)}) - \inf g\bigr).
\end{equation}
Using boundedness of iterates and the strong convexity of $\mathcal A$ on bounded sets
we deduce 
\begin{equation}
\frac{\tau}{n}\bigl(g(\theta^{(0)}) - \inf g\bigr)
\ge
K_1 \min_{1\le k\le n} \|\log(x^{(k)}) - \log(x^{(k+1)})\|^2.
\end{equation}
Since the iterates $(x^k)_{k\in\N}$ are bounded away from $0$ and $+\infty$, there exist constants $0<a\le b$ such that
\begin{equation}
x_i^{(k)},\, x_i^{(k+1)} \in [a,b]
\quad \text{for all } i,k.
\end{equation}
On $[a,b]$, the function $\log$ is Lipschitz and strongly monotone, hence bi-Lipschitz. In particular, there exist constants $K_2, K_2'>0$ such that for all $x,y \in [a,b]^n$,
\begin{equation}
K_2 \|x-y\|^2 \le \|\log x - \log y\|^2 \le K_2' \|x-y\|^2.
\end{equation}
Thus we deduce that
\begin{equation}
\frac{\tau}{n}\bigl(g(\theta^{(0)}) - \inf g\bigr) 
\ge 
K_1 K_2 \min_{1\le k\le n} \|x^{(k)} - x^{(k+1)} \|^2
\end{equation}
Taking $K = K_1 K_2 \tau$, we obtain the desired result.
\end{proof}

\section{Constrained EMML} \label{sec:Constrained_EMML}

Interpreting the EMML algorithm as a MD algorithm allows incorporating constraints while preserving convergence guaranties. Let 
$C \subset \R^m$ be a convex set such that $C \cap \R^n_{>0} \neq \emptyset$.

We consider the constrained optimization problem
\begin{equation}
(\mathcal{P}_C) \quad \min_{x \in \R_{\geq 0}^m} \KL(y, A x) 
\quad \text{subject to} \quad x \in C.
\end{equation}

As shown in Proposition~\ref{prop:EMML_as_MD}, EMML is equivalent to MD applied to the variable $\theta = \log(x)$. A natural way to enforce constraints is thus to apply projected MD in the $\theta$-variable. Define the constraint set in the dual variables by
\begin{equation}
C_\theta := \{ \theta \in \R^m : e^\theta \in C \}.
\end{equation}
We assume that $C_\theta$ is closed and convex, which holds for sets of the form
$C = \{x \in \mathbb{R}^n : r(x) \leq 0\}$,
where $r : \mathbb{R}^n \to \mathbb{R}$ is continuous and such that the mapping
$\theta \mapsto r(e^\theta)$ is convex. In particular, this condition is satisfied by
$r(x) = \sum_{i=1}^n x_i - 1$.
Denoting as $\iota_{C_\theta}$ the indicator function of $C_\theta$, we  target
\begin{equation} \label{eq:min_cEMML}
\min_\theta F(\theta) \coloneq \KL(y, A e^\theta) + \iota_{C_\theta}(\theta),
\end{equation}
where in \eqref{eq:min_cEMML}, we identify $g(\theta) = \mathrm{KL}(y, A e^\theta)$ and $f(\theta) = \iota_{C_\theta}(\theta)$.
The Bregman proximal gradient update \eqref{eq:MD_Bregman_prox}, with respect to the Legendre function $\A$, can then be rewritten for this specific choice of functions as:
\begin{equation}
\theta^{(r+1)} = \operatorname{prox}_{\iota_{C_\theta}}^{\A}
\Bigl( (\nabla \A)^{-1} \bigl( \nabla \A(\theta^{(r)}) - \tau \nabla \KL(y, A e^{\theta^{(r)}}) \bigr) \Bigr).
\end{equation}

Recall that the Bregman proximal operator of an indicator function reduces to a Bregman projection \cite{BREGMAN1967, censor_parallel_1998}:
\begin{equation} \label{eq:prox_iota}
\operatorname{prox}_{\iota_{C_\theta}}^{\A}(\tilde \theta) 
= \argmin_{\theta \in C_\theta} D_{\A}(\theta, \tilde  \theta),
\end{equation}
where $D_\A$ is the Bregman divergence associated with $\A$. In practice, for many constrained sets defined for the primary variable $x$ (e.g., linear constraints), computing the Bregman projection in the $\theta$-variable may be difficult. 
\begin{proposition}
Let $\theta^\star$ denote the Bregman projection \eqref{eq:prox_iota} of $\tilde \theta$ and define
$ x^\star = e^{\theta^\star} $, $ \tilde x = e^{\tilde \theta}$.
Let
\begin{equation}
w_j := \sum_i A_{i,j} > 0, 
\;\; 
\tilde{w}_j := w_j \tilde{x_j}, 
\; \; 
\phi(x) = - \sum_j \tilde{w}_j \log x_j.
\end{equation}
Then, solving~\eqref{eq:prox_iota} is equivalent to the Bregman projection
\begin{equation}~
x^\star = \operatorname{Proj}_{C}^{\phi}(\tilde x) \coloneq \argmin_{x \in C} D_\phi(x, \tilde x),
\end{equation}
where $D_\phi$ is the Bregman divergence associated with $\phi$.
\end{proposition}

\begin{proof}
The problem \eqref{eq:prox_iota} can be rewritten as
\begin{equation}\label{eq:bregman_proj}
x^\star \in \argmin_{x \in \R_{>0}^n} D_\A(\log x, \log \tilde  x) \quad \text{s.t.} \quad x \in C.
\end{equation}
A direct computation shows that, up to a constant independent of $x$, the Bregman divergence $D_{\A}(\log x, \log \tilde x)$ writes
\begin{equation}\label{eq:rewrite_DA}
D_{\A}(\log x, \log \tilde  x) \propto \sum_{j} w_j x_j - \sum_{j} w_j \tilde x_j \log x_j,
\end{equation}
which coincides, up to constants, with the Bregman divergence \sego{$D_\phi(x, \tilde{x})$}. 
The Bregman projection in the $\theta$-variable is thus equivalent to the Bregman projection of $x$ onto $C$ with divergence $D_\phi$.
\end{proof}

The resulting constrained EMML algorithm is summarized in Algorithm~\ref{algo:C_EMML_new_with_proj}.
\begin{algorithm}[htbp]
\caption{Constrained EMML via Projected MD} \label{algo:C_EMML_new_with_proj}
\textbf{Input:} $A, y, M, y$ ; initial point $x^{(0)} \in \R_{>0}^m$; weights $w_j := \sum_i A_{i,j}$. 

\begin{algorithmic}[1]
\For{$r = 0, 1, 2, \dots$}
\State $\tilde{x}^{(r)} = x^{(r)} \odot \dfrac{A^\top (y / (A x^{(r)}))}{A^\top \mathbf{1}}$ \Comment{EMML update}
\State Define $\tilde{w} = w \odot \tilde{x}^{(r)}$ and $\phi(x) = -\tilde{w}^\top \log(x)$
\State $x^{(r+1)} = \operatorname{Proj}_{C}^{\phi}(\tilde{x}^{(r)})$ \Comment{Bregman projection}
\EndFor
\State \textbf{return} $x^{(r+1)}$
\end{algorithmic}
\end{algorithm}

\begin{remark}
Note that the mirror map $\phi$ used in the algorithm depends on the iteration index. More precisely, both the reference point $\tilde{w}$ and thus the mirror map $\phi$ vary with the iteration $r$, i.e.,
\[
\tilde{w} = \tilde{w}^{(r)}, \qquad \phi = \phi^{(r)}.
\]
However, to simplify the notation and improve readability, we omit the superscript $(r)$.
\end{remark}

\begin{proposition}\label{prop:cv_speed_EMML_c}
Assume that every row and every column of $A$ contains at least one strictly positive entry, and that $A^\top y >0$. further assume that iterates of Algorithm \ref{algo:C_EMML_new_with_proj} stay bounded. Then Algorithm~\ref{algo:C_EMML_new_with_proj} converges in the sense of Proposition \ref{prop:cvMD_sequence}.
Moreover, similarly to Proposition \ref{prop:cv_speed_EMML}, for any $N \in \mathbb{N}^*$, any initial point $x^0 \in \R^n_{>0}$, 
there exists a constant $K$ such that the iterates of Algorithm~\ref{algo:C_EMML_new_with_proj} satisfy
\begin{equation}
\min_{\substack{1 \le k \le N \\ x \in C}} \|x^{(k)} - x^{(k+1)}\|^2 \leq \frac{K}{N}\bigl(g(\theta^{(0)}) - \inf g\bigr).
\end{equation}
\end{proposition}

\begin{proof}
\blue
Points 1 and 2 of Assumption~\ref{assumption:MD_nonLipschitz_1} are established exactly as in the proof of Proposition~\ref{prop:cv_speed_EMML}.
Points 3 and 4 are immediate from the choice
\begin{equation}
f(\theta)=\iota_{C_\theta}(\theta),
\end{equation}
since $C_\theta$ is assumed to be nonempty and convex. The beginning of the verification of Point 5 follows the same arguments as in the proof of Proposition~\ref{prop:cv_speed_EMML}, but we instead consider the function
\begin{equation}
z \mapsto
\iota_{C_\theta}(z)
+
\langle z-\theta,\nabla g(\theta)\rangle
+
\frac{1}{\tau}D_\phi(z,\theta).
\end{equation}
Since $C_\theta$ is nonempty and closed, this function is proper, lower semicontinuous, and coercive. Therefore, it admits at least one minimizer.
Moreover, the minimizer is unique. Indeed, the mapping $z\mapsto D_\phi(z,\theta)$ is strictly convex, and adding a linear term together with the convex indicator function $\iota_{C_\theta}$ preserves strict convexity on $C_\theta$. Consequently, $T_\tau(\theta)$ is nonempty and single-valued.
Finally, since $\dom \phi=\mathbb{R}^n$, every minimizer necessarily belongs to
$
\operatorname{int}(\dom \phi)=\mathbb{R}^n.
$
Hence, $T_\tau$ maps $\operatorname{int}(\dom \phi)$ into itself and Assumption~\ref{assumption:MD_nonLipschitz_2} is verified exactly as in the proof of Proposition~\ref{prop:cv_speed_EMML}.

We may therefore apply Propositions~\ref{prop:cvMD_sequence} and~\ref{prop:cvMD_iter} to obtain the claimed convergence result. The remainder of the proof is identical to that of Proposition~\ref{prop:cv_speed_EMML}, as it does not depend on the particular choice of~$f$.
\end{proof}

\section{Numerical Experiments} \label{sec:numerical_results}

We consider hyperspectral unmixing \cite{bioucas2012hyperspectral} to compare the classical EMML algorithm with iterations given by \eqref{eq:EMML_update} and its constrained version in Algorithm \ref{algo:C_EMML_new_with_proj}. 
Hyperspectral images record hundreds of contiguous spectral bands for each pixel of a scene, and due to the limited spatial resolution of the sensors, each pixel typically contains a mixture of spectral responses from several materials. 
Hyperspectral unmixing aims at decomposing each observed pixel spectrum into a set of pure spectral signatures, called endmembers, and estimating their associated abundance fractions \cite{zhang2016advancement}. 
The abundance vector $x = (x^q)^\top$, $q \in \{\text{Soil}, \text{Tree}, \text{Water}\}$, represents the spatial abundance maps. Estimating $x$ from $y$ and $A$ constitutes the hyperspectral unmixing problem \cite{chan2018two}. We compare
\begin{align*}
&\text{Classical EMML:}
&&\min_{x \ge 0} \; \KL(y, Ax), \\[0.5ex]
&\text{Constrained EMML:}
&&\min_{x \ge 0} \; \KL(y, Ax)
\quad \text{s.t.} \quad \sum_j x_j \leq 1 , 
\end{align*}
where the simplex constraint empirically enforces that the abundances for each pixel sum to one \cite{iordache2012total, chouzenoux_fast_2014, chouzenoux_proximal_2020}. 
Reconstruction quality is assessed using the Peak Signal-to-Noise Ratio (PSNR).
Algorithms are run until $\|x^{(k+1)} - x^{(k)}\|_2^2 / \|x^{(k)}\|_2^2 < 10^{-5}$ or a maximum of $N = 1000$ iterations is reached.

We first empirically assess the convergence speed of (constrained) EMML and compare it with the theoretical bound for EMML given in \cite{kunstner2021homeomorphic}. 
The results are reported in Figure~\ref{fig:convergence_comparison}. The constrained algorithm converges faster than the unconstrained one. Consistent with the observations in \cite{kunstner2021homeomorphic}, both EMML variants exhibit convergence rates that are significantly faster than the sublinear rate predicted by the theoretical bound.

We then assess the quality of the reconstructions.
For all noise levels, the reconstructed abundance maps produced by both algorithms are visually similar, as shown in Figure~\ref{fig:results_abundance}.
Nevertheless, the reconstructions obtained with the unconstrained EMML appear noticeably noisier than those produced by the constrained EMML.
This qualitative observation is supported by a quantitative evaluation: the constrained EMML consistently achieves more accurate reconstructions in fewer iterations, yielding higher PSNR values across all noise levels, as reported in Figure~\ref{fig:PSNR_abundance}.

\begin{figure}[htbp]
\centering

\begin{minipage}[t]{0.49\linewidth}
    \centering
    \includegraphics[
        width=\linewidth,
        trim=0 0.3cm 0 0.2cm,
        clip
    ]{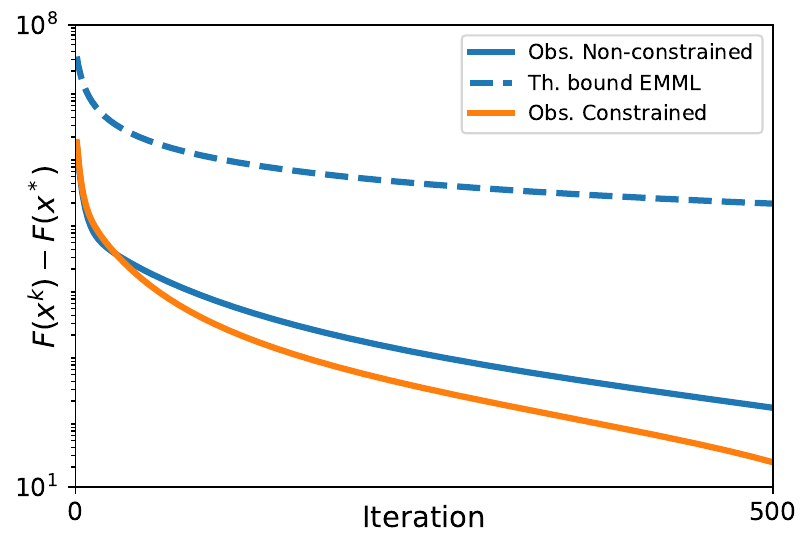}
    \captionof{figure}{Comparison between the observed convergence behavior and the corresponding sublinear theoretical bound for EMML with $\sigma = 40$, where $F(x^\star)$ denotes the minimum value of $F$. In both cases, the empirical convergence is substantially faster than the predicted sublinear rate.}
    \label{fig:convergence_comparison}
\end{minipage}
\hfill
\begin{minipage}[t]{0.49\linewidth}
    \centering
    \includegraphics[
        width=\linewidth,
        trim=0 0.3cm 0 0.2cm,
        clip
    ]{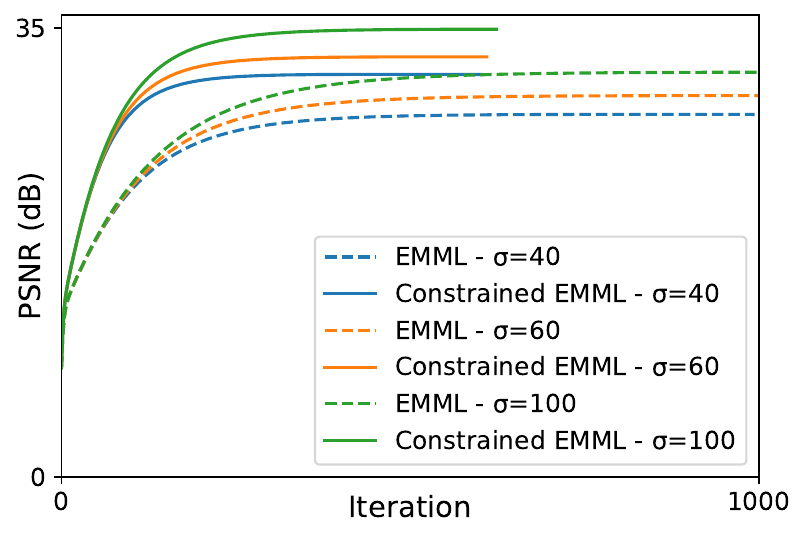}
    \captionof{figure}{PSNR of the reconstructed abundance maps obtained with EMML and constrained EMML at different noise levels. The constrained algorithm consistently achieves higher PSNR values and outperforms EMML even at higher noise levels, while converging in nearly half as many iterations.}
    \label{fig:PSNR_abundance}
\end{minipage}

\end{figure}

\begin{figure}[htbp]
\centering
\begin{tabular}{@{}c@{\hskip 0mm}c@{\hskip 0mm}c@{}}
\begin{tikzpicture}[spy using outlines={rectangle, orange, magnification=2, size=2cm}] 
\node (img) {\includegraphics[width=0.32\linewidth, trim=0 0 0 0.65cm, clip]{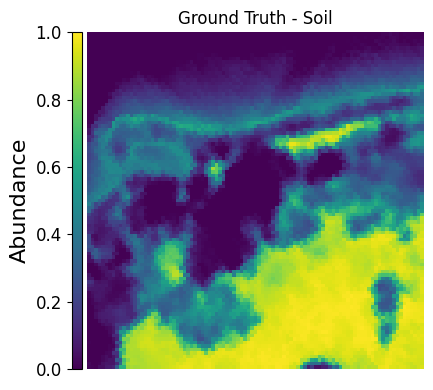}};
\spy on ($(img.center)+(-1,0.8)$) 
    in node [anchor=south east] at ($(img.south east)+(-0.05,0.05)$);
\end{tikzpicture}
&
\begin{tikzpicture}[spy using outlines={rectangle, orange, magnification=2, size=2cm}] 
\node (img) {\includegraphics[width=0.32\linewidth, trim=0 0 0 0.65cm, clip]{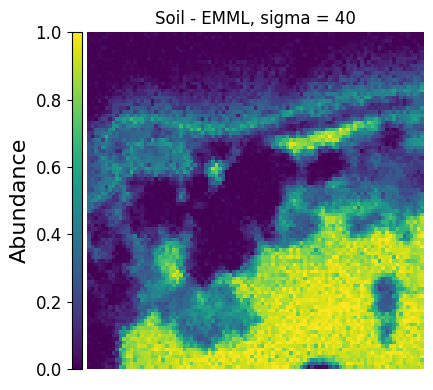}};
\spy on ($(img.center)+(-1,0.8)$) 
    in node [anchor=south east] at ($(img.south east)+(-0.05,0.05)$);
\end{tikzpicture}
&
\begin{tikzpicture}[spy using outlines={rectangle, orange, magnification=2, size=2cm}] 
\node (img) {\includegraphics[width=0.32\linewidth, trim=0 0 0 0.65cm, clip]{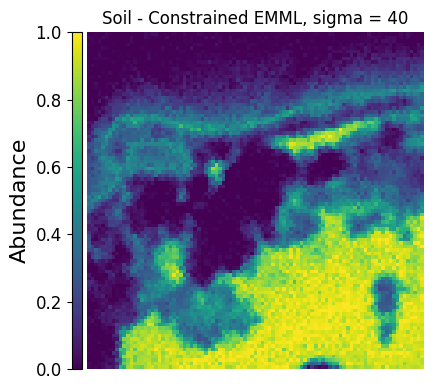}};
\spy on ($(img.center)+(-1,0.8)$) 
    in node [anchor=south east] at ($(img.south east)$);
\end{tikzpicture}

\\
(a) Ground Truth
&
(b) EMML~\eqref{eq:EMML_update}
&
(c) Constrained 
\\
&&EMML~\eqref{algo:C_EMML_new_with_proj}

\end{tabular}

\caption{Reconstructed abundance map for Soil with Poisson noise $\sigma = 40$. The reconstructions obtained with the EMML algorithm are visibly noisier than those obtained with the Constrained EMML algorithm.}
\label{fig:results_abundance}
\end{figure}

\section{Conclusion and Perspectives} \label{sec:Conclusion}

We revisited EMML as a MD method on a reparametrized objective, which allowed the introduction of convex constraints while preserving convergence guarantees. Numerical experiments show that these constraints accelerate convergence and improve reconstruction quality by acting as a regularization.

The method relies on Bregman projections, which can be challenging beyond simplex constraints \cite{martin2013iterative, constrainedEMML}. Extending it to more general convex or linear constraints would further broaden its applicability in inverse problems.

\section*{Acknowledgments}

AC acknowledges funding from the Berlin Mathematical School and the French Research Agency through the Holibrain project (ANR-23-CE45-0020). AC was also supported by the France 2030 BPI project Cone Beam AI. SM work was funded by the Deutsche Forschungsgemeinschaft (DFG, German Research Foundation) under Germany's Excellence Strategy – The Berlin Mathematics
Research Center MATH+ (EXC-2046/1, project ID: 390685689).

Many thanks to Chun-Neng Chu for pointing out that $f(\theta) = \KL(y,Ae^\theta)$ is in general not convex.

{\fontsize{10}{10}\selectfont
\bibliography{references}
}
\end{document}